\documentclass{amsart}
\usepackage{mathrsfs}

\usepackage{amsfonts,amssymb}
 \usepackage{xspace,xcolor}
 \usepackage[all,cmtip]{xy}
 \usepackage{graphicx,enumerate}
 \usepackage{tikz}
 \usetikzlibrary{cd}
 \usepackage{ifthen}

\usepackage{lscape}
\usepackage{color}
\usepackage[colorlinks,citecolor=blue]{hyperref}
\usepackage[useregional]{datetime2}

\newcommand{\cl}[1]{\marginpar{\color{black}\tiny #1 --cl}} 

\theoremstyle{plain}
\newtheorem{theorem}{Theorem}[section]
\newtheorem{thm}{Theorem}[section]
\newtheorem{lemma}[theorem]{Lemma}
\newtheorem{corollary}[theorem]{Corollary}
\newtheorem{conjecture}[theorem]{Conjecture}
\newtheorem{proposition}[theorem]{Proposition}
\newtheorem{prop}[theorem]{Proposition}
\newtheorem{prop-defn}[theorem]{Proposition-Definition}

\newtheorem*{theorem:main}{Main Theorem} 
\theoremstyle{definition}
\newtheorem{defn}[theorem]{Definition}

\newtheorem{remark}[theorem]{Remark}

\newtheorem{question}[theorem]{Question}
\newtheorem*{remark*}{Remark}
\newtheorem*{remarks*}{Remarks}

\newcommand{\C}{\mathcal C}
\newcommand{\Aff}{{\rm{Aff}}}

\newcommand{\PSL}{{\rm{PSL}}}

\newcommand{\GL}{{\rm{GL}}}

\newcommand{\SL}{\rm SL}

\newcommand{\SAff}{{\rm SAff}}

\newcommand{\Hom}{\rm{Hom}}

\usepackage{mathabx}

\title{Fibered $3$--manifolds and Veech groups}
\author[Leininger]{Christopher J. Leininger}
	\address{Department of Mathematics, Rice University, Houston, TX}
	\email{cjl12@rice.edu}
\author[Rafi]{Kasra Rafi}
    \address{Department of Mathematics, University of Toronto, Ontario, Canada}
    \email{rafi@math.toronto.edu}
\author[Rouse]{Nicholas Rouse}
	\address{Department of Mathematics, University of British Columbia, Vancouver, BC Canada}
	\email{rouse@math.ubc.ca}
\author[Shinkle]{Emily Shinkle}
	\address{Department of Mathematics, University of Illinois, Urbana-Champaign, IL}
	\email{esshinkle@gmail.com}
\author[Verberne]{Yvon Verberne}
    \address{Department of Mathematics, University of Western Ontario, Ontario, Canada}
    \email{verberne.math@gmail.com}
    
\date{\today}
\begin{document}

\begin{abstract}
We study Veech groups associated to the pseudo-Anosov monodromies of fibers and foliations of a fixed hyperbolic 3-manifold.  Assuming Lehmer’s Conjecture, we prove that the Veech groups associated to fibers generically contain no parabolic elements.  For foliations, we prove that the Veech groups are always elementary.
\end{abstract}

\maketitle

\section{Introduction}

A pseudo-Anosov homeomorphism $f\colon S\to S$ on an orientable surface determines a complex structure and holomorphic quadratic differential, $(X,q)$, up to Teichm\"uller deformation, for which the vertical and horizontal foliations are the stable and unstable foliations of $f$.  The pseudo-Anosov generates an infinite cyclic subgroup of the full group of orientation preserving affine homeomorphisms, $\Aff_+(X,q)$.

For a finite type surface $S$, we say that the pseudo-Anosov homeomorphism $f$ is {\em lonely} if $\langle f \rangle < \Aff_+(S,q)$ has finite index.  The motivation for this paper is the following; see e.g.~Hubert-Masur-Schmidt-Zorich \cite{HuMaScZo} and Lanneau~\cite{Lanneau}

\begin{conjecture} [Lonely p-As] There exist lonely pseudo-Anosov homeomorphisms.  In fact, lonely pseudo-Anosov homeomorphisms are generic.
\end{conjecture}

There is not an agreed upon notion of ``generic'', and some care must be taken:  work of Calta \cite{Calta} and McMullen \cite{McGenus2,McTeichTrace} shows that {\em no} pseudo-Anosov homeomorphism on a surface of genus $2$, with orientable stable/unstable foliation is lonely.  In fact, in this case, not only are the pseudo-Anosov homeomorphisms not lonely, but their Veech groups always contain parabolic elements.

In this paper, we consider infinite families of pseudo-Anosov homeomorphism arising as follows; see \S\ref{S:fibered 3-manifolds}.  Suppose $f \colon S \to S$ is a pseudo-Anosov homeomorphism of a finite type surface $S$ and $M_f$ the mapping torus (which is hyperbolic by Thurston's Hyperbolization Theorem \cite{Otal-ThurstonHyp}). The connected cross sections of the suspension flow are organized by their cohomology classes (up to isotopy), which are primitive integral classes in the cone on the open fibered face $F \subset H^1(M,\mathbb R)$ of the Thurston norm ball containing the Poincar\'e-Lefschetz dual of the fiber $S$.  Given such an integral class $\alpha$, the first return map to the cross section $S_\alpha$ is a pseudo-Anosov homeomorphism $f_\alpha \colon S_\alpha \to S_\alpha$.  When $b_1(M) > 1$, there are infinitely many such pseudo-Anosov homeomorphisms; in fact, $|\chi(S_\alpha)|$ is a linear function of $\alpha$, and hence tends to infinity with $\alpha$.  

We let $\bar \alpha \in F$ denote the projection of the primitive integral class $\alpha$ in the cone over $F$, and let $F_{\mathbb Q}$ be the set of all such projections, which is precisely the (dense) set of rational points in $F$.

\begin{question} \label{Q:lonely fibers}
Given a fibered hyperbolic $3$--manifold and fibered face $F$, are the pseudo-Anosov homeomorphism $f_\alpha$ for $\bar \alpha \in F_{\mathbb Q}$ generically lonely?
\end{question}

We will provide two pieces of evidence that the answer to this question is `yes'. Write $\Aff_+(X_\alpha,q_\alpha)$ for the orientation preserving affine group containing $f_\alpha$; see \S\ref{S:Veech groups} for more details.

\newcommand{\TParabolicsA}
{Suppose $F$ is the fibered face of an orientable, fibered, hyperbolic $3$--manifold.  Assuming Lehmer's Conjecture, the set of $\bar \alpha \in F_{\mathbb Q}$  such that $\Aff_+(X_\alpha,q_\alpha)$ contains a parabolic element is discrete in $F$.}
\begin{theorem} \label{T:locally finite parabolics}
\TParabolicsA
\end{theorem}
In certain examples, the set of classes whose associated Veech group contains parabolics is actually finite (again, assuming Lehmer's conjecture); see Theorem~\ref{T:finite parabolics}.  In \S\ref{S:examples} we describe some explicit computations that illustrate this finite set.  If $M$ is the orientation cover of a non-orientable, fibered $3$--manifold, then the conclusion of Theorem~\ref{T:locally finite parabolics} holds on the invariant cohomology of the covering involution without assuming the validity of Lehmer's Conjecture; see Theorem~\ref{T:non-orientable}.


Much of the defining structure survives for non-integral classes $\alpha \in F - F_{\mathbb Q}$; see \S\ref{S:foliations in cone} for details.  Briefly, we first recall that every $\alpha \in F - F_{\mathbb Q}$ is represented by a closed $1$--form $\omega_\alpha$ which is positive on the vector field generating the suspension flow.  The kernel of $\omega_\alpha$ is tangent to a foliation $\mathcal F_\alpha$, and the flow can be reparameterized to send leaves of $\mathcal F_\alpha$ to other leaves.  There is no longer a first return time, but rather a {\em higher rank abelian group} of return times, $H_\alpha$, to any given leaf $S_\alpha$ of $\mathcal F_\alpha$. 
Work of McMullen \cite{Mc} associates a {\em leaf-wise} complex structure and quadratic differential $(X_\alpha,q_\alpha)$ to each $\alpha \in F - F_{\mathbb Q}$ so that the leaf-to-leaf maps of the flow are all Teichm\"uller maps.  For every leaf $S_\alpha$ of $\mathcal F_\alpha$, the return maps to $S_\alpha$ thus determine an isomorphism from $H_\alpha < \mathbb R$ to a subgroup we denote $H_\alpha^{\Aff} \!\! < \Aff_+(X_\alpha,q_\alpha)$, an abelian group of pseudo-Anosov elements.  Our second piece of evidence for a positive answer to Question~\ref{Q:lonely fibers} is the following.

\newcommand{\TLonely}{{If $F$ is a fibered face of a closed, orientable, fibered, hyperbolic $3$--manifold, then for all $\alpha \in F- F_{\mathbb Q}$, and any leaf $S_\alpha$ of $\mathcal F_\alpha$, the abelian group $H_\alpha^{\Aff} \!\! < \Aff_+(X_\alpha,q_\alpha)$ has finite index.}}
\begin{theorem} \label{T:lonely leaves}
\TLonely
\end{theorem}

For $\alpha \in F-F_{\mathbb Q}$, the leaves $S_\alpha$ are infinite type surfaces.  In general, there is much more flexibility in constructing affine groups for infinite type surfaces, and exotic groups abound.  Indeed, work of Przyticki-Schmithusen-Valdez \cite{PrScVa} and Ram\'{\i}rez-Valdez \cite{RaVa} proves that {\em any} countable subgroup of $\GL_2(\mathbb R)$ without contractions is the derivative-image of some affine group. (See also Bowman \cite{Bowman-lonely} for a ``naturally occurring" lonely pseudo-Anosov homeomorphism on an infinite type surface of finite area.) Theorem~\ref{T:lonely leaves} says that for the leaves $S_\alpha$ of the foliations and their associated quadratic differentials, the situation is much more rigid.

\subsection*{Acknowledgements} The authors would like to thank Erwan Lanneau, Livio Liechti, Alan Reid, and Ferr\'an Valdez for helpful conversations, and to the anonymous referee for suggestions that improved the readability.  We are particularly grateful to Liechti for suggesting Theorem~\ref{T:non-orientable}.  The first author was partially supported by NSF grant DMS-2106419. The second author was partially supported by NSERC Discovery grant, RGPIN 06486. 
The fifth author was partially supported by an NSERC-PDF Fellowship.

\section{Definitions and background}

\subsection{Fibered \texorpdfstring{$3$}{3}--manifolds} \label{S:fibered 3-manifolds}

Here we explain the set up and background for our work in more detail.  For a pseudo-Anosov homeomorphism $f \colon S\to S$ of an orientable, finite type surface $S$,  let $\lambda(f)$ denote its {\em stretch factor} (also called its {\em dilatation}); see \cite{flp:TTS}.  We write
\[ M = M_f = S \times [0,1]/(x,1) \sim (f(x),0)\]
to denote the mapping torus of the pseudo-Anosov homeomorphism $f$.  The suspension flow $\psi_s$ of $f$ is generated by the vector field $\xi = \frac{\partial}{\partial t}$, where $t$ is the coordinate on the $[0,1]$ factor.  Alternatively, we have the local flow of the same name $\psi_s(x,t) = (x,t+s)$ on $S \times [0,1]$, defined for $t, s+t \in [0,1]$, which descends to the suspension flow.  

A {\em cross section} (or just {\em section}) of the flow is a surface  $S_\alpha \subset M$ transverse to $\xi$, such that for all $x \in S_\alpha$, $\psi_s(x) \in S_\alpha $ for some $s >0$.  If $s(x) > 0$ is the smallest such number, then the {\em first return map} of $\psi_s$ is the map $f_\alpha \colon S_\alpha \to S_\alpha$ defined by $f_\alpha(x) = \psi_{s(x)}(x)$ for $x \in S_\alpha$.  Note that $S (= S \times \{0\}) \subset M$ is a section, and the first return map to $S$ is precisely the map $f = \psi_1|_S$.

Cutting open along an arbitrary section $S_\alpha$ we get a product $S_\alpha \times [0,1]$ where the slices $\{x\} \times [0,1]$ are arcs of flow lines.  Thus, $M$ can also be expressed as the mapping torus of $f_\alpha$, or alternatively, $M$ fibers over the circle with {\em monodromy} $f_\alpha$.  Up to isotopy, the fiber $S_\alpha$ is determined by its Poincar\'e-Lefschetz dual cohomology class $\alpha = [S_\alpha] \in H^1(M; \mathbb Z) \subset H^1(M;\mathbb R) = H^1(M)$.  To see how these are organized, we first recall the following theorem of Thurston \cite{ThNorm}

\begin{theorem} \label{T:Thurston cone} For $M = M_f$ as above, there is a finite union of open, convex, polyhedral cones $\mathcal C_1,\ldots,\mathcal C_k \subset H^1(M)$ such that $\alpha \in H^1(M;\mathbb Z)$ is dual to a fiber in a fibration over $S^1$ if and only if $\alpha \in \mathcal C_j$ for some $j$.  Moreover, there is a norm $\| \cdot \|_T$ on $H^1(M)$ so that for each $\mathcal C_j$, $\| \cdot \|_T$ restricted to $\mathcal C_j$ is linear, and if $\alpha \in \mathcal C_j \cap H^1(M;\mathbb Z)$ then $\|\alpha\|_T$ is the negative of the Euler characteristic of the fiber dual to $\alpha$.
\end{theorem} 
The unit ball $\mathfrak B$ of $\| \cdot \|_T$ is a polyhedron, and each $\mathcal C_j$ is the cone over the interior of a top dimensional face $F_j$ of $\mathfrak B$.

The cones in the theorem are called the {\em fibered cones} of $M$ and the $F_j$ the {\em fibered faces} of $\mathfrak B$.  It follows from Thurston's proof of Theorem~\ref{T:Thurston cone} that each of the sections $S_\alpha$ of $(\psi_s)$ described above must lie in a single one of the fibered cones $\mathcal C$ over a fibered face $F$.  The following theorem elaborates on this, combining results of Fried from \cite{Fr,Fr0}.

\begin{theorem} \label{T:Fried cone} For $M = M_f$ as above, there is a fibered cone $\mathcal C \subset H^1(M)$ such that $\alpha \in H^1(M;\mathbb Z)$ is dual to a section of $(\psi_s)$ if and only if $\alpha \in \mathcal C$.  Moreover, there is a function $\mathfrak h \colon \mathcal C \to \mathbb R_+$ which is continuous, convex, and homogenous of degree $-1$, with the following properties.
\begin{itemize}
\item For any $\alpha \in \mathcal C \cap H^1(M;\mathbb Z)$, $f_\alpha$ is pseudo-Anosov and $\mathfrak h(\alpha) = \log(\lambda(f_\alpha))$.
\item For any $\{ \alpha_n\} \subset \mathcal C$ with $\alpha_n \to \partial \mathcal C$, we have $\mathfrak h(\alpha_n) \to \infty$;
\end{itemize} 
\end{theorem}
We let $\mathcal C_{\mathbb Z} \subset \mathcal C$ denote the primitive integral classes in the fibered cone $\mathcal C$; that is, the integral points which are not nontrivial multiples of another element of $H^1(M;\mathbb Z)$.  These correspond precisely to the connected sections of $(\psi_s)$.

McMullen \cite{Mc} refined the analysis of $\mathfrak h$, proving for example that it is actually real-analytic.  For this, he computed the stretch factors using his \textit{Teichm\"uller polynomial} $\Theta_{\mathcal{C}}$. This polynomial 
\[\Theta_{\mathcal{C}} = \sum_{g\in G}a_g g\]
is an element of the group ring $\mathbb{Z}[G]$ where $G=H_1(M;\mathbb{Z})/\text{torsion}$. For $\alpha \in \mathcal{C}_{\mathbb Z}$, the \textit{specialization} of the Teichm\"uller polynomial is 
\[\Theta_{\mathcal{C}}^\alpha(t) = \sum_{g\in G} a_g t^{\alpha(g)} \in \mathbb{Z}[t^{\pm 1}]\]
where we view $\alpha \in H^1(M;\mathbb Z) \cong \Hom(G;\mathbb Z)$.
Further, $G \cong H \oplus \mathbb{Z}$
where $H=\text{Hom}(H^1(S,\mathbb{Z})^f,\mathbb{Z}) \cong \mathbb{Z}^m$ and $H^1(S,\mathbb{Z})^f$ are the $f$--invariant cohomology classes. So we can regard $\Theta_{\mathcal{C}}$ as a Laurent polynomial on the generators $x_1, x_2, \ldots, x_m$ of $H$ and the generator $u$ of $\mathbb{Z}$. Then specialization to the dual of an element $(a_1, a_2, \ldots, a_m, b) \in \mathcal{C} \cap H^1(M;\mathbb{Z})$ amounts to setting $x_i=t^{a_i}$ for $1\leq i \leq m$ and $u=t^b$. 
McMullen proves that the specializations and the pseudo-Anosov first return maps are related by the following.

\begin{theorem} \label{T:Teich poly}
For any $\alpha \in \C_{\mathbb Z}$, the stretch factor $\lambda(f_\alpha)$ is a root of $\Theta_{\mathcal{C}}^{\alpha}$ with the largest modulus.
\end{theorem}

Combining the linearity of $\| \cdot \|_T$ on $\mathcal C$ together with the homogeneity of $\mathfrak h$, we have the following observation of McMullen; see \cite{Mc}.
\begin{corollary} The function $\alpha \mapsto \|\alpha\|_T \mathfrak h(\alpha)$ is continuous and constant on rays from $0$.  In particular, if $K \subset \mathcal C$ is any compact subset, then $\| \cdot \|_T \mathfrak h(\cdot)$ is bounded on $\mathbb R_+ K$.
\end{corollary}

The key corollary for us is the following, also observed by McMullen from the same paper.
\begin{corollary} \label{C:asymptotics} If $\{ \alpha_n\}_n \subset \mathcal C_{\mathbb Z}$ is any infinite sequence of distinct elements, then $|\chi(S_{\alpha_n})| \to \infty$ and if the rays $\mathbb R_+ \alpha_n$ do not accumulate on $\partial \mathcal C$, then
\[ \log(\lambda(f_{\alpha_n})) \asymp \frac{1}{|\chi(S_{\alpha_n})|}. \] 
In particular, $\lambda(f_{\alpha_n}) \to 1$.
\end{corollary}

\begin{remark}One can sometimes promote the final conclusion to {\em any} infinite sequence of distinct elements, without the assumption about non-accumulation to $\partial \mathcal C$; see the examples in \S\ref{S:examples}.  This is not always the case, and the accumulation set of stretch factors can be fairly complicated, as described by work of Landry-Minsky-Taylor \cite{LaMiTay}. 
\end{remark}

\subsection{Foliations in the fibered cone} \label{S:foliations in cone}

Fried's work described above \cite{Fr,Fr0} implies that any $\alpha \in \mathcal C$ may be represented by a closed $1$--form $\omega_\alpha$ for which $\omega_\alpha(\xi)>0$ at every point of $M$.  For integral classes, $\omega_\alpha$ is the pull-back of the volume form from the fibration over the circle $\mathbb R/\mathbb Z$, and in general, $\omega_\alpha$ is a convex combination of such $1$--forms.  The kernel of $\omega_\alpha$ defines a foliation $\mathcal F_\alpha$ transverse to $\xi$ whose leaves are injectively immersed surfaces $S_\alpha \subset M$.   We consider the reparameterized flow $\{\psi_s^\alpha\}$ defined by scaling the generating vector field $\xi$ by $\xi/\omega_\alpha(\xi)$.  Then for every leaf $S_\alpha \subset M$ of $\mathcal F_\alpha$ and for every $s \in \mathbb R$, the image by the flow $\psi_s^\alpha(S_\alpha)$ is another leaf of $\mathcal F_\alpha$.  The subgroup $H_{\alpha} < \mathbb R$ mentioned in the introduction is precisely the set of return times of $\psi_s^\alpha$ to $S_\alpha$.  As such, $H_\alpha$ acts on $S_\alpha$ so that $s \in H_\alpha$ acts by $s \cdot x = \psi_s^\alpha(x)$, for all $x \in S_\alpha$. 

The group $H_\alpha \cong \mathbb Z^n$ for some $n= n_{\alpha} \leq b_1(M)$, and can alternatively be defined as the set of periods of $\alpha$ (i.e.~the $\alpha$--homomorphic image of $H_1(M;\mathbb Z)$).  A leaf $S_\alpha$ is a closed surface, and in fact a fiber as above if and only if $n_\alpha = 1$ in which case $H_\alpha$ is a discrete subgroup of $\mathbb R$ and $\bar \alpha \in F_{\mathbb Q}$. 
On the other hand, $n_\alpha \geq 2$ if and only if the group of return times $H_\alpha$ is indiscrete, and so $S_\alpha$ is {\em dense} in $M$.


\subsection{Teichm\"uller flows and Veech groups} \label{S:Veech groups}

In \cite{Mc}, McMullen defines a conformal structure and quadratic differential, $(X_\alpha,q_\alpha)$, on the leaves $S_\alpha$ of the foliation $\mathcal F_\alpha$, for all $\alpha \in \C$, with the following properties.  For each $s \in \mathbb R$ and leaf $S_\alpha$, the leaf-to-leaf map $\psi_s^\alpha \colon S_\alpha \to \psi_s^\alpha(S_\alpha)$ is a Teichm\"uller map with initial/terminal quadratic differentials given by $q_\alpha$ on the respective leaves.  In fact, there exists some $K_\alpha > 1$ so that $\psi_s^\alpha$ is a $K_\alpha^{|s|}$--Teichm\"uller map, and hence $K_\alpha^{2|s|}$--quasi-conformal.  

\begin{remark} The notation $(X_\alpha,q_\alpha)$ is somewhat ambiguous: this really denotes a family of structures, one on every leaf, though we abuse notation and also use this same notation to denote the restriction to any given leaf.
\end{remark}

The vertical and horizontal foliations of $q_\alpha$ on the leaves $S_\alpha$ of $\mathcal F_\alpha$ are obtained by intersecting with a {\em fixed} singular foliation on the $3$--manifold; namely, the suspension of the unstable/stable foliations for the original pseudo-Anosov homeomorphism $f$.  In particular, the cone points (i.e.~zeros) of $q_\alpha$ are precisely the intersections of $S_\alpha$ with the $\psi_s$--flowlines through the cone points on the original surface $S$.  Consequently, the cone points are isolated, and the cone angles are bounded by those of the original surface, and are hence bounded independent of $\alpha$.

For $s \in H_\alpha$, $\psi_s^\alpha \colon S_\alpha \to S_\alpha$ is (a remarking) of the Teichm\"uller map, and thus an affine pseudo-Anosov homeomorphism with respect to $q_\alpha$.  In this way, we obtain an isomorphism from $H_\alpha$ to a subgroup $H_\alpha^{\Aff} < \Aff_+(X_\alpha,q_\alpha)$, the group of orientation preserving affine homeomorphisms of the leaf $S_\alpha$ with respect to $(X_\alpha,q_\alpha)$.  The derivative with respect to the preferred coordinates defines a map
\[ D_\alpha \colon \Aff_+(X_\alpha,q_\alpha) \to \GL_2^+(\mathbb R)/\pm I,\]
which is called the {\em Veech group} of $(X_\alpha,q_\alpha)$.
A {\em parabolic} element of $\Aff_+(X_\alpha,q_\alpha)$ is one whose image by $D_\alpha$ is parabolic.

\begin{remark} The preferred coordinates for a quadratic differential are only defined up to translation and rotation through angle $\pi$, so the derivative is only defined up to sign.  If all affine homeomorphisms are area preserving (e.g.~if the surface has finite area) then the derivative maps to $\PSL_2(\mathbb R) = \SL_2(\mathbb R)/\pm I$.
\end{remark}
Since the vertical/horizontal foliations are the stable/unstable foliations, the image of $H_\alpha^\Aff$, which we denote $H_\alpha^D = D_\alpha(H_\alpha^\Aff)$ is contained in the diagonal subgroup of $\PSL_2(\mathbb R)$:
\[ H_\alpha^D < \Delta = \left\{ \left. \left( \begin{array}{cc} a & 0 \\ 0 & \frac1a \end{array}  \right) \in \SL_2(\mathbb R) \,\, \right| \, \, a > 0 \, \, \right\}/\pm I.\]

Define $\SAff(X_\alpha,q_\alpha) < \Aff_+(X_\alpha,q_\alpha)$ to be the area preserving subgroup of orientation preserving affine homeomorphisms; this is the preimage of $\PSL_2(\mathbb R)$ under $D_\alpha$.   In particular, $H_\alpha^\Aff < \SAff(X_\alpha,q_\alpha)$.

\subsection{Trace fields}
A number field is \textit{totally real} if the image of every embedding into $\mathbb{C}$ lies in $\mathbb{R}$. Hubert-Lanneau~\cite{HuLaParabolic} proved the following.

\begin{theorem}\label{theorem: totally real trace field if parabolics} If a nonelementary Veech group contains a parabolic element, then the trace field is totally real.
\end{theorem}

A pseudo-Anosov $f$ being lonely implies that there are no parabolic elements in the Veech group, but not conversely; see \cite{HuLaMo}.

McMullen~\cite[Corollary~9.6]{McTeichTrace} proved the following fact about the trace field of a Veech group; see also Kenyon-Smillie~\cite{KenSmi}.

\begin{theorem} \label{T:pA trace generates} The trace field of a Veech group containing a pseudo-Anosov is generated by the trace of that pseudo-Anosov.  That is, the trace field is given by $\mathbb Q(\lambda(f) + \lambda(f)^{-1})$.
\end{theorem}

Thus, this trace field is totally real precisely when the trace of the pseudo-Anosov has only real Galois conjugates.

\begin{remark}
    Theorems~\ref{theorem: totally real trace field if parabolics} and \ref{T:pA trace generates} are proved for complex structures with an abelian differential, rather than a quadratic differential.  The proof of Theorem~\ref{theorem: totally real trace field if parabolics} for the more general case of quadratic differentials follows verbatim since the key ingredient is the so-called Thurston-Veech construction, which works for both quadratic differentials and abelian differentials (see \cite[\S6]{Thurston-geomDynSurfs}).  Theorem~\ref{T:pA trace generates} for quadratic differentials follows from the case of abelian differentials since every affine homeomorphism lifts to the canonical $2$--fold cover where a quadratic differential pulls back to a square of an abelian differential, and thus the preimage of the Veech group of the original surface in $\SL_2(\mathbb R)$ is contained in the Veech group for the abelian differential.
\end{remark}

\subsection{Lehmer's Conjecture}

Theorem~\ref{T:locally finite parabolics} is dependent on the validity of what is known as Lehmer's conjecture \cite{Lehmer} though Lehmer did not actually conjecture the statement we will use. See \cite{SmythSurvey}.  To state this conjecture, we need the following.
\begin{defn}
	Let $p(x) \in \mathbb C[x]$ with factorization over $\mathbb C$
	\[
		p(x) = a_0\prod_{i=1}^{m}(x-\gamma_i).
	\] 
	The \textbf{Mahler measure} of $p$ is
	\[
	\mathcal{M}(p) = \left|a_0\right|\prod_{i=1}^{m}(\max{1, |\gamma_i|}).
	\]
\end{defn}
With this definition, we state the conjecture we assume.
\begin{conjecture}[Lehmer] \label{conj:Lehmer}
	There is a constant $\mu > 1$ such that for every $p(x) \in \mathbb{Z}[x]$ with a root not equal to a root of unity $\mathcal{M}(p) \geq \mu$.
\end{conjecture}

Lehmer's Conjecture is known in some special cases, including the following result of Schinzel \cite{Schinzel} which will be important in the proof of Theorem~\ref{T:non-orientable}.

\begin{theorem} \label{T:Schinzel}
If $p(t)$ is the minimal polynomial for an algebraic integer not equal to $0$ or $\pm 1$, all of whose roots are real, then
\[ \mathcal M(p) \geq \left(\tfrac{1+\sqrt{5}}2\right)^{\deg(p)/2}.\]
\end{theorem}




\section{Examples} \label{S:examples}
Here we provide examples of fibered faces of fibered 3-manifolds and examine arithmetic features of the Veech groups of the corresponding pseudo-Anosov homeomorphisms. 

\subsection{Example 1} \label{ex:hironaka1}
Let $\beta = \sigma_1\sigma_2^{-1}$ be an element of the braid group $B_3$ on three strands (viewed as the mapping class group of a four-punctured sphere, $S$), where $\sigma_1$ and $\sigma_2$ denote the standard generators.  Let $M$ denote the mapping torus of $\beta$. McMullen computes the Teichm\"uller polynomial for this manifold in detail in \cite{Mc}. See also Hironaka \cite{Hironaka}. 

Since $\beta$ permutes the strands of the braid cyclically, $b_1(M)=2$. Choosing appropriate bases, we obtain an isomorphism $H^1(M;\mathbb{Z}) \cong \mathbb{Z}^2$ so that the starting fiber surface $S$ is dual to $(0,1)$, the fibered cone is
\[ \mathcal{C} = \{(a,b)\in \mathbb{R}^2 : b > 0, -b < a < b\}\]
and the Teichm\"uller polynomial for this cone is 
\[\Theta_{\mathcal{C}}(x,u) = u^2 - u(x + 1 + x^{-1}) - 1.\]
Specialization to an integral class $(a,b) \in \mathcal{C}_{\mathbb Z}$ equates to setting $x=t^a$ and $u=t^b$ and yields
\[\Theta_{\mathcal C}^{(a,b)}(t) = \Theta_{\mathcal{C}}(t^a,t^b) = t^{2b}-t^{b+a}-t^{b}-t^{b-a}+1.\]

We used the mathematics software system SageMath \cite{sage} to factor $\Theta_{\mathcal C}^{(a,b)}(t)$ for all primitive integral pairs $(a,b) \in \mathcal{C}$ with $b < 50$, to determine the stretch factors $\lambda_{(a,b)}$ of the corresponding monodromies and their minimal polynomials. We then computed the conjugates of the corresponding traces, $\lambda_{(a,b)}+1/\lambda_{(a,b)}$, to determine whether the trace field of each associated Veech group is totally real. The results are shown in Figure \ref{figure: hironaka cone}. Recall that by Theorem \ref{theorem: totally real trace field if parabolics}, when this trace field is not totally real, the Veech group has no parabolic elements. 

These computations suggest that there are only finitely many pairs $(a,b)$ where the trace field is not totally real.  This is not a coincidence as we will see below.  For this, we record the following improvement on Corollary \ref{C:asymptotics} for the cone $\mathcal C$ for this example.
\begin{lemma} \label{L:Hironaka finite}
For any sequence $\alpha_n = (a_n,b_n) \in {\mathcal C}_{\mathbb Z}$ of distinct elements, we have $\lambda(f_{\alpha_n}) \to 1$.
\end{lemma}
\begin{proof} 
Since $\mathfrak h$ is convex, the maximum value of $\mathfrak h(a,b) = \log(\lambda(f_{(a,b)}))$, for points $(a,b) \in \mathcal C_{\mathbb Z}$ and a fixed $b$, occurs at either $(b-1,b)$ or $(1-b,b)$.

First we consider the points of the form $(b-1,b)$.  The specialization of $\Theta_{\mathcal C}$ in this case takes the form
\[ \Theta_{\mathcal C}^{(b-1,b)}(t) = t^{2b} - t^{2b-1} - t^b - t+1.\]
Recall that $\lambda_b  = \lambda(f_{(b-1,b)}) > 1$. As $b \to \infty$, we claim that $\lambda_b \to 1$. Suppose instead that the sequence is bounded below by $1+\epsilon$, for $\epsilon > 0$ on some subsequence. Then in this subsequence we have
\begin{align*}
    \Theta_{\mathcal C}^{(b-1,b)}(\lambda_b)    &=  \lambda_b^{2b}(1 - \lambda_b^{-1} - \lambda_b^{-b} - \lambda_b^{1-2b}) + 1 \\
        &\geq   (1+\epsilon)^{2b}\left(1-(1+\epsilon)^{-1} - (1+\epsilon)^{-b} - (1+\epsilon)^{1-2b}\right)
\end{align*}
The first factor on the right hand side tends to infinity when $b$ does, while the second factor tends toward $1-(1+\epsilon)^{-1} = \epsilon / (1+\epsilon) > 0$. This implies that $\Theta_{\mathcal C}^{(b-1,b)}(\lambda_b)$ approaches infinity, whereas instead it is identically equal to 0. This contradiction proves the claim.


For points of the form $(1-b,b)$, the specialization takes the form
\[ \Theta_{\mathcal C}^{(1-b,b)}(t) = t^{2b} - t - t^b - t^{2b-1} + 1 = \Theta_{\mathcal C}^{(b-1,b)}(t). \]
Therefore, $\lambda(f_{(1-b,b)}) = \lambda(f_{(b-1,b)}) = \lambda_b$ and as $b \to \infty$, these both tend to $1$.
\end{proof}

\begin{figure}
\includegraphics[width=\linewidth, trim = {0 2cm 0 3cm}, clip]{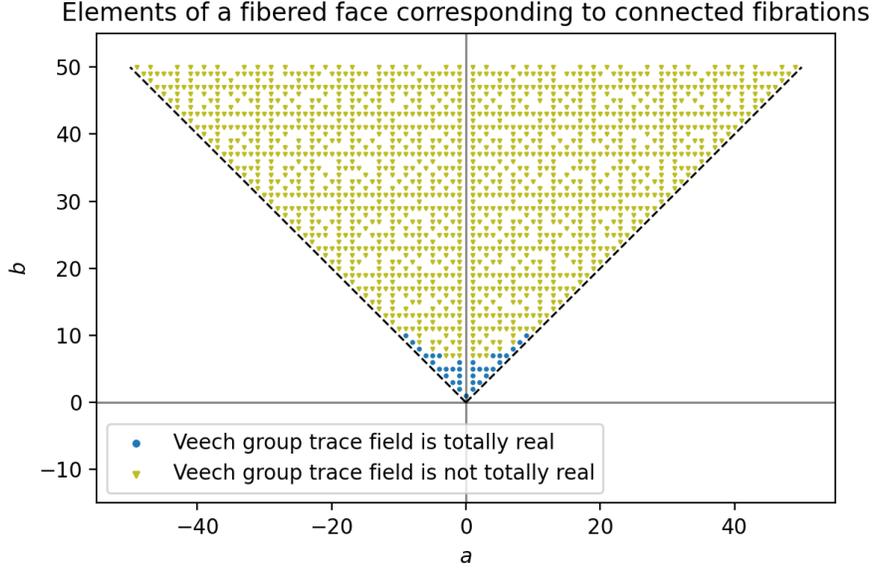}
\caption{Primitive integral elements in a fibered cone for the mapping torus of the three-strand braid $\sigma_1\sigma_2^{-1}$. Elements marked with green triangles have corresponding Veech group with trace field that is not totally real.}
\label{figure: hironaka cone}
\end{figure}

One of the difficulties in the proof of Theorem~\ref{T:locally finite parabolics} is understanding the degrees of the trace field.  This is complicated by the fact that the Teichm\"uller polynomial need not be irreducible in general. For example, when specialized to $(a,b) = (9, 14)$, the Teichm\"uller polynomial in this example splits into the cyclotomic polynomials $t^2 - t + 1$ and $t^4 - t^2 + 1$, plus the minimal polynomial of the corresponding stretch factor. However, in other cases, such as the specialization to $(a,b) = (5,14)$, the Teichm\"uller polynomial remains irreducible. We refer the reader to \cite{FilGar} for more on the factorizations of the specialized polynomials in the example above. 
As we will see in the example below, the Teichm\"uller polynomial also sometimes admits additional non-cyclotomic factors aside from the minimal polynomial of the corresponding stretch factor. 

\subsection{Example 2}
Let $\beta' = \beta^2$, for $\beta$ from the preceding example. Let $M'$ denote the mapping torus on $\beta'$ and ${\theta'}_{\mathcal{C}'}$ the Teichm\"uller polynomial of the fibered cone $\mathcal{C}'$ containing the dual of $\beta'$. Here we will observe three different splitting behaviors of specializations of the Teichm\"uller polynomial. In particular, we see that certain specializations of ${\theta'}_{\mathcal{C}'}$ split into multiple non-cyclotomic factors, limiting what information can be derived about conjugates of the corresponding stretch factors and their traces by looking at the collection of all roots of ${\theta'}_{\mathcal{C}'}$. 

The Teichm\"uller polynomial here is
\[{\theta'}_{\mathcal{C}'}(x,u) = u^2 - u(x^2 + 2x + 1 +2x^{-1} + x^{-2}) + 1\]
over the cone 
\[\mathcal{C} = \{(a,b) \in \mathbb{R}^2 : b > 0, -b/2 < a < b/2\}.\] 
The specialization to $(a,b) = (6,17)$ is irreducible over $\mathbb{Z}$: 
\[t^{34} - t^{29} - 2t^{23} - t^{17} - 2t^{11} - t^5 + 1,\]
while the specialization to $(a,b) = (7, 17)$ splits as a cyclotomic and non-cyclotomic factor:
\begin{multline*}
    (t^4 + t^3 + t^2 + t + 1)  (t^{30} - t^{29} - t^{27} + t^{26} + t^{25} - t^{24} - t^{22} + t^{21} - t^{20} + t^{19} - t^{17} + t^{16}\\
    - t^{15} + t^{14} - t^{13} + t^{11} - t^{10} + t^9 - t^8 - t^6 + t^5 + t^4 - t^3 - t + 1),
\end{multline*}
and the specialization to $(a,b) = (7, 18)$ has multiple non-cyclotomic factors:
\[(t^2 - t + 1) (t^4 + t^3 + t^2 + t + 1) (t^{12} - t^9 - t^8 + t^7 + t^6 + t^5 - t^4 - t^3 + 1) (t^{18} - t^{16} - t^9 - t^2 + 1).\] 
Figure \ref{figure: fibered cone squared} shows whether the Veech groups corresponding to elements of $\mathcal{C}'$ have totally real trace field. For all three specializations described in this example, the corresponding Veech group trace field is not totally real.

The analog to Lemma \ref{L:Hironaka finite} holds in this example as well. $M'$ is a 2-fold cover of $M$ so the stretch factors in $\mathcal{C}_{\mathbb{Z}}'$ are at most squares of the stretch factors in $\mathcal{C}_{\mathbb{Z}}$.


\begin{figure}
\includegraphics[width=\linewidth, trim = {0 2cm 0 3cm}, clip]{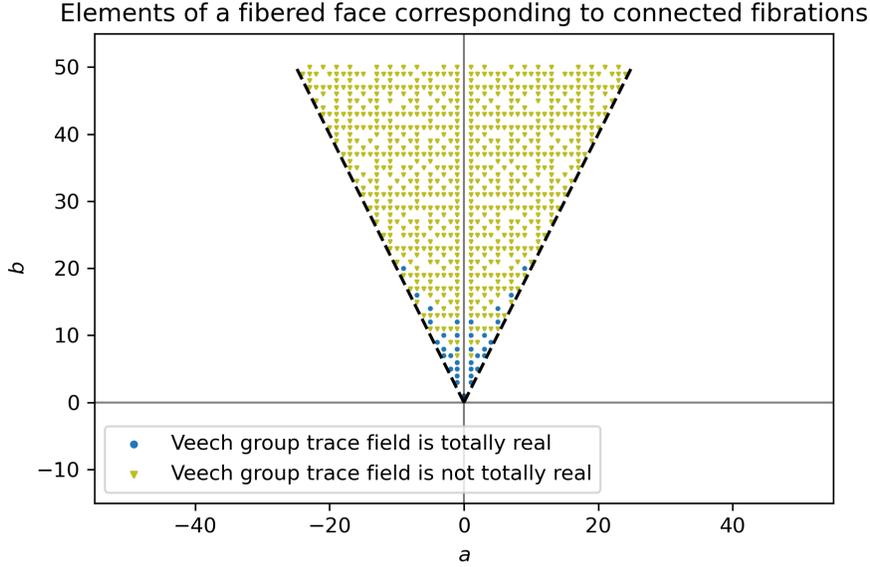}
\caption{Primitive integral elements in a fibered cone for the mapping torus of the three-strand braid $(\sigma_1\sigma_2^{-1})^2$. Elements marked with green triangles have a not totally real corresponding Veech group.}
\label{figure: fibered cone squared}
\end{figure}

\section{Most Veech groups have no parabolics}

We are now ready for the proof of the first theorem from the introduction.\\

\noindent
{\bf Theorem~\ref{T:locally finite parabolics}.} {\em \TParabolicsA}

\begin{proof} Consider any sequence of distinct elements $\alpha_n$ in $\mathcal C_{\mathbb Z}$ such that $\bar \alpha_n$ does not accumulate on $\partial F$.  
We need to show that $\Aff(X_\alpha,q_{\alpha_n})$ contains a parabolic for at most finitely many $n$.  According to Theorem~\ref{theorem: totally real trace field if parabolics}, it suffices to prove that the trace field is totally real for at most finitely many $n$.
Setting $\lambda_n = \lambda(f_{\alpha_n})$, Theorem~\ref{T:pA trace generates} implies that the trace field of $\Aff(X_{\alpha_n},q_{\alpha_n})$ is $\mathbb Q(\lambda_n + \lambda_n^{-1})$.

Next, let $N$ be the number of terms of the Teichm\"uller polynomial, $\Theta_{\mathcal C}$ for $\mathcal C$.
The stretch factor $\lambda_n$ is the largest modulus root of the specialization $\Theta_{\mathcal C}^{\alpha_n}(t)$ by Theorem~\ref{T:Teich poly}.  We observe that this polynomial has no more nonzero terms than $\Theta_{\mathcal C}$, and thus has at most $N$ terms.
Descartes's rule of signs implies that the number of real roots of $\Theta_{\mathcal C}^{\alpha_n}$ is at most $2N-2$.

Suppose that $p_n(t)$ is the minimal polynomial of $\lambda_n$, which is thus a factor of $\Theta_{\mathcal C}^{\alpha_n}(t)$ (up to powers of $t$, which we will ignore).  In particular, note that $\lambda_n$ bounds the modulus of all other roots of $p_n(t)$. The stretch factors are always algebraic integers, and hence $p_n(t)$ is monic.  The Mahler measure is therefore the product of the moduli of the roots outside the unit circle.  There are at most $2N-2$ real roots of $\Theta_{\mathcal C}^{\alpha_n}(t)$, and hence the same is true of $p_n(t)$.  Write
\[ \mathcal M(p_n) = A_nB_n \]
where $A_n$ is the product of the moduli of the {\em real roots} and $B_n$ is the product of the moduli of the non-real roots outside the unit circle (and $1$ if there are none). Thus, we have
\begin{equation} \label{E:real root bound}
    A_n \leq \lambda_n^{2N-2}.
\end{equation}

Now, as $n \to \infty$, we have $|\chi(S_{\alpha_n})| = \| \alpha_n \|_T \to \infty$ as $n \to \infty$.  Since $\bar \alpha_n$ does not accumulate on $\partial F$, Corollary~\ref{C:asymptotics} implies $\lambda_n = \lambda(f_{\alpha_n}) \to 1$. By \eqref{E:real root bound}, it follows that $A_n \to 1$ as $n \to \infty$.  Since we are assuming Lehmer's Conjecture, it follows that $B_n > 1$ for all but finitely many $n$.  That is, there is at least one non-real root $\zeta_n$ of $p_n(t)$ outside the unit circle.  (In fact, the number of such roots tends to infinity linearly with $|\chi(S_{\alpha_n})|$ since $\lambda_n$ has the maximum modulus of any root of $p_n(t)$).

Therefore, for all but finitely many $n$, the embedding of $\mathbb Q(\lambda_n + \lambda_n^{-1})$ to $\mathbb C$ sending $\lambda_n+\lambda_n^{-1}$ to $\zeta_n + \zeta_n^{-1}$ has non-real image, since $\zeta_n$ is non-real and lies off the unit circle.  Therefore, $\mathbb Q(\lambda_n+\lambda_n^{-1})$ is totally real for at most finitely many $n$, as required.
\end{proof}

\begin{remark} The proof of Theorem~\ref{T:locally finite parabolics} follows a strategy of Craig Hodgson, \cite{Hodgson}, for understanding trace fields under hyperbolic Dehn filling.
\end{remark}

The key ingredient is that for sequences $\{\alpha_n\}$ in $\C_{\mathbb Z}$, we have $\lambda(f_{\alpha_n}) \to 1$.  Sometimes this happens for any sequence of distinct elements in the cone, and then one obtains the following stronger result.

\begin{theorem} \label{T:finite parabolics}
Suppose $F$ is the fibered face of an orientable, fibered, hyperbolic $3$--manifold and that $1$ is the only accumulation point of the set \[ \{\lambda(f_\alpha) \mid \bar \alpha \in F_{\mathbb Q}\}.\]
Assuming Lehmer's Conjecture, the set of $\bar \alpha \in F_{\mathbb Q}$  such $\Aff_+(X_\alpha,q_\alpha)$ contains a parabolic element is finite. 
\end{theorem}

\begin{proof} This is exactly the same as the proof of Theorem~\ref{T:locally finite parabolics}, except that the assumption that $1$ is the only accumulation point of $\{\lambda(f_\alpha) \mid \bar \alpha \in F_{\mathbb Q}\}$ replaces the references to Corollary~\ref{C:asymptotics}, and does away with the requirement that $\bar \alpha_n$ does not accumulate on $\partial F$.
\end{proof}

Returning to the examples from Section \ref{S:examples}, Lemmas~\ref{L:Hironaka finite} and the discussion in both implies that the hypotheses of Theorem~\ref{T:finite parabolics} are satisfied.  Thus only finitely many elements $\alpha \in \mathcal C_{\mathbb Z}$ are such that $\Aff_+(X_\alpha,q_\alpha)$ can contain parabolics.  We refer the reader to \cite{LaMiTay} for more on accumulation set of $\{\lambda(f_\alpha) \mid \alpha \in \C_{\mathbb Z}\}$


If $p \colon M \to N$ is the orientation double cover of a non-orientable fibered $3$--manifold $N$ with covering involution $\tau \colon M \to M$, then $p^* \colon H^1(N) \to H^1(M)$ is an isomorphism onto the $\tau^*$--fixed subspace.  There is a well-defined Thurston norm on $H^1(N)$, and the induced homomorphism $\pi_1N \to \pi_1S^1 = \mathbb Z$ determines an element $\alpha \in H^1(N)$ which lies in an open cone of a fibered face.  Indeed, the $p^*$--image of this cone is the intersection of $p^*(H^1(N))$ with an open cone on a fibered face $F$ for $M$, or equivalently, the cone over the $\tau^*$--fixed set $F^{\tau} \subset F$; see \cite[Theorem~2.11]{KhPaWi}.  In this setting, and appealing to work of Liechti and Strenner, \cite{LieStr}, we can remove the assumption that Lehmer's Conjecture holds, at the expense of restricting to $F^\tau$.
\begin{theorem} \label{T:non-orientable}
With assumptions above on $M \to N = M/\langle \tau \rangle$, the set of $\bar \alpha \in F^\tau_{\mathbb Q}$ such that $\Aff_+(X_\alpha,q_\alpha)$ contains a parabolic element is discrete in $F^\tau$.
\end{theorem}
\begin{proof}
For every $\bar \alpha \subset F^\tau_{\mathbb Q}$, the associated monodromy $f_\alpha \colon S_\alpha \to S_\alpha$ is the lift of the monodromy for some fibration of $N$.  Then either $S_\alpha$ covers a non-orientable surface $S_\alpha'$ and $f_\alpha$ is the lift of a pseudo-Anosov homeomorphism on $S_\alpha'$, or else $f_\alpha$ is the square of an orientation reversing pseudo-Anosov homeomorphism.  In either case, \cite[Theorem~1.10]{LieStr} implies that if $p(t)$ is the minimal polynomial for $\lambda(f_\alpha)$, then $p(t)$ has no roots on the unit circle.

Now suppose $\{\bar\alpha_n\} \subset F^\tau_{\mathbb Q}$ is any infinite sequence of distinct elements not accumulating on the boundary of $F$ and $\lambda_n = \lambda(f_{\alpha_n})$.  As in the proof of Theorem~\ref{T:locally finite parabolics}, write $p_n(t)$ for the minimal polynomial and $\mathcal M(p_n) = A_nB_n$.  Again, $A_n \to 1$, and thus by Theorem~\ref{T:Schinzel}, there is a non-real root $\zeta_n$ of $p_n(t)$ for all $n$ sufficiently large (regardless of the behavior of $B_n$). By the previous paragraph $\zeta_n$ is not on the unit circle, and thus $\zeta_n + \zeta_n^{-1} \not \in \mathbb C$, and hence $\mathbb Q(\lambda_n +\lambda_n^{-1})$ is not totally real, proving our result.
\end{proof}

\section{Veech groups of leaves}
We now turn our attention to the non-integral points in the cone and the second theorem from the introduction.

\bigskip

\noindent {\bf Theorem~\ref{T:lonely leaves}.} {\em \TLonely
}  

\bigskip

For the rest of the paper, we assume $M$ is a closed, fibered, hyperbolic $3$--manifold.  The results of this section are only nontrivial if $b_1(M) >1$, since otherwise $F- F_{\mathbb Q} = \emptyset$ for any fibered face $F$ (since in that case $F = F_{\mathbb Q}$ is a point).
Given $\alpha \in F$, we recall that $\psi_s^\alpha$ is the reparameterized flow as in \S\ref{S:foliations in cone}, that sends leaves of $\mathcal F_\alpha$ to leaves. 
 Furthermore, $(X_\alpha,q_\alpha)$ is the leaf-wise conformal structure and quadratic differential, and there is $K_\alpha > 1$ so that $\psi_s^\alpha$ is the $K_\alpha^{|s|}$--Teichm\"uller map, hence $K_\alpha^{2|s|}$--quasi-conformal and $K_\alpha^{|s|}$--bi-Lipschitz.

\begin{lemma} \label{L:compact flow} For any $\alpha \in F - F_{\mathbb Q}$ there exists a compact subsurface $Z \subset S_\alpha$ such that
\[ M = \bigcup_{s \in [0,1]} \psi_s^\alpha(Z).\]
\end{lemma}
\begin{proof} Choose an exhaustion of $S_\alpha$ by a sequence of compact subsurfaces: 
\[Z_1 \subsetneq Z_2 \subsetneq Z_3 \subsetneq \cdots S_\alpha,  \mbox{ and } \bigcup_{n=1}^\infty Z_n = S_\alpha,
\]
and observe that
\[ \left\{ \bigcup_{s \in (0,1)} \psi_s^\alpha(\mbox{int}(Z_n)) \right\}_{n=1}^\infty \]
is an open cover of $M$ since every leaf is dense.
Since $M$ is compact, the open cover admits a finite subcover of $M$.
As the compact surfaces $Z_n$ are nested, there exists an index $N$ such that for $Z = Z_N$ we have
\[
M = \bigcup_{s \in [0,1]} \psi_s^\alpha(Z). \qedhere
\]
\end{proof}

The isomorphism $H_\alpha \cong H_\alpha^\Aff$ is given by $s \mapsto \psi_s^\alpha|_{S_\alpha}$.
We write
\[ H^\Aff_\alpha[0,1]\subset H^\Aff_\alpha\] for the image of $H_\alpha \cap [0,1]$ under this isomorphism.   Note  that every element of $H^\Aff_\alpha$ is $K_\alpha^2$--quasi-conformal and $K_\alpha$--bi-Lipschitz since $s \leq 1$.  As a consequence of Lemma~\ref{L:compact flow}, we have the following.

\begin{corollary} \label{C:compact translates} For $\alpha \in F- F_{\mathbb Q}$ and $Z \subset S_\alpha$ as in Lemma~\ref{L:compact flow} we have
\[ S_\alpha = \bigcup_{h \in H_\alpha^\Aff[0,1]} h(Z).\]
\end{corollary}
\begin{proof}
Let $Z \subset S_\alpha$ be the compact subsurface from Lemma \ref{L:compact flow}, so that for every $x \in S_{\alpha} \subseteq M$, we have $x \in \psi_s^\alpha(Z)$ for some $s \in [0,1]$.
Since $x \in S_{\alpha}$, this implies that $s \in H_{\alpha}$. Therefore
\[ 
S_\alpha = \bigcup_{s \in H_\alpha \cap [0,1]} \psi_s^\alpha(Z) = \bigcup_{h \in H_\alpha^\Aff[0,1]} h(Z). \qedhere
\]
\end{proof}

\begin{corollary} \label{C:bounded geometry} For any $\alpha \in F-F_{\mathbb Q}$ there exists $C >0$ so that for any leaf $S_\alpha$ of $\mathcal F_\alpha$, the geometry of $q_\alpha$ is bounded. Specifically, (1) there is a lower bound on the length of any saddle connection, in particular a lower bound on the distance between any two cone points, (2) all cone points have finite (uniformly bounded) cone angle, and (3) $(X_\alpha,q_\alpha)$ is complete. 
\end{corollary}

\begin{proof}
Let $S_{\alpha}$ be any leaf, and consider the compact surface $Z$ from Corollary \ref{C:compact translates}.
By making $Z$ slightly larger, we can assume that no singular points of $q_{\alpha}$ lie on the boundary of $Z$.
Denote the set of all singularities of $q_{\alpha}$ by $A$.
Let $d_{\partial Z}(a)$ denote the distance of a singularity $a \in A$ to the boundary of $Z$, and let $d_{Z}(a, b)$ denote the minimal length of a saddle connection in $Z$ between two (not necessarily distinct) singularities $a,b \in A \cap Z$.
Since $Z$ is compact, we have that
\[
\epsilon = \min \left\{ \min_{a, b \in A \cap Z} d_{Z}(a, b), \min_{a \in A} d_{\partial Z}(a) \right\} > 0.
\]
Pick a saddle connection $\omega$ connecting any singularity $a$ to any singularity $b$. 
There exists an $h \in H_\alpha^\Aff[0,1]$ such that $h(Z)$ contains $a$.
Since $h$ is $K_{\alpha}$--bi-Lipschitz, either $\omega$ is contained in $h(Z)$ and has length at least $\epsilon K_\alpha^{-1}$, or it leaves $h(Z)$ and we again deduce that $\omega$ has length at least the distance from $a$ to $\partial h(Z)$, which is at least $\epsilon K_\alpha^{-1}$.  In either case, we obtain a uniform lower bound $\epsilon K_\alpha^{-1}$ to the length of $\omega$, proving (1).

As was noted in Section \ref{S:Veech groups}, we have that all cone points have finite cone angle which proves (2).
Since $Z$ is compact, there is an $\epsilon'$ so that the $\epsilon'$--neighborhood of $Z$ also has compact closure, which is thus complete.  Any Cauchy sequence has a tail that is contained in the $h$-image of the closure of this neighborhood for some $h \in H_\alpha^\Aff[0,1]$.  Since this $h$--image is also complete, the Cauchy sequence converges, and we have that $(X_\alpha,q_\alpha)$ is complete which proves (3).
\end{proof}

\begin{remark}\label{Rmk:TameSurfaces}
    Note that Corollary~\ref{C:bounded geometry} implies that our surfaces are tame in the sense of Definition 2.1 of \cite{PrScVa}.
\end{remark}

An important observation is the following: for any element of $g \in \Aff_+(X_\alpha,q_\alpha)$, we can choose some element $h \in H^\Aff_{\alpha}[0,1]$ so that $h \circ g(Z) \cap Z \neq \emptyset$, and furthermore, if $g$ is $K$--quasi-conformal, then $h\circ g$ is $(KK_\alpha^2)$--quasi-conformal.

\begin{proposition} \label{P:constant subsequences} Suppose $\alpha \in F-F_{\mathbb Q}$, $K_0 > 1$, and $\{g_n\}_{n=1}^\infty \subset \Aff_+(X_\alpha,q_\alpha)$ is a sequence of elements with $K(g_n) \leq K_0$. Then there is a subsequence $\{g_{n_k}\}_{k=0}^\infty$ and $\{h_{n_k}\}_{k=0}^\infty \subset H_{\alpha}^\Aff[0,1]$ so that $h_{n_k} \circ g_{n_k} = h_{n_0} \circ g_{n_0}$ for all $k \geq 0$.

\end{proposition} \label{P:}
\begin{proof} From the observation before the statement, we can find $h_n \in H_{\alpha}^\Aff[0,1]$ so that $h_n \circ g_n(Z) \cap Z \neq \emptyset$.  Next, observe that $h_n \circ g_n$ is $(K_0K_{\alpha}^2)$--quasi-conformal, so by compactness of quasi-conformal maps, after passing to a subsequence, $h_{n_k} \circ g_{n_k}$ converges uniformly on compact sets to a map $f$.  The maps $h_{n_k} \circ g_{n_k}$ are affine, so they must map cone points to cone points.  Since the cone points are uniformly separated by Corollary~\ref{C:bounded geometry}, there are a pair of cone points $a,b$ so that for $k$ sufficiently large $h_{n_k} \circ g_{n_k}(a) = b$.  Moreover, if we pick a pair of saddle connections in linearly independent directions emanating from $a$, then for $n$ sufficiently large $h_{n_k} \circ g_{n_k}$ all agree on this pair, again by Corollary~\ref{C:bounded geometry}.  But these conditions uniquely determines the affine homeomorphism, and hence $h_{n_k} \circ g_{n_k}$ is eventually constant, and passing to a tail-subsequence of this subsequence completes the proof.
\end{proof}

From this we can prove a special case of Theorem~\ref{T:lonely leaves}:

\begin{proposition}\label{Prop:H_alphaFI} If $\alpha \in F-F_{\mathbb Q}$, then $H_{\alpha}^\Aff$ has finite index in $\SAff(X_\alpha,q_\alpha)$.
\end{proposition}
\begin{proof} 
Suppose $H_{\alpha}^{\Aff}$ is not finite index, consider the closure of the $D_\alpha$--image in $\PSL_2(\mathbb R)$:
\[ G = \overline{D_\alpha(\SAff(X_\alpha,q_\alpha))}. \]
Since $\alpha \in F- F_{\mathbb Q}$, 
every leaf $S_\alpha$ of $\mathcal F_\alpha$ is dense in $M$.  Therefore $H_{\alpha}^{D}< \Delta \cong \mathbb R$ is an abelian subgroup with rank at least $2$, and hence is dense in $\Delta$.  Consequently, $\Delta < G$.  

By the classification of Lie subalgebras of $\mathfrak{s}\mathfrak{l}_2(\mathbb R)$ (or a direct calculations) we observe that, after replacing $G$ with a finite index subgroup, we must be in one of the following situations:
\begin{enumerate}
\item $G = \PSL_2(\mathbb R)$,
\item $G$ is the subgroup of upper triangular matrices, or
\item $G = \Delta$.
\end{enumerate}
In any case, we claim that there is a sequence of elements $\{g_n\} \subset \SAff(X_\alpha,q_\alpha)$ such that $D_\alpha(g_n) \to I$ in $\PSL_2(\mathbb R)$ and so that $H_{\alpha}^{\Aff}g_n$ are distinct cosets of $H_{\alpha}^{\Aff}$.  Assuming the claim, we prove the proposition.  For this, we simply apply Proposition~\ref{P:constant subsequences}, pass to a subsequence (of the same name) so that $h_n \circ g_n = h_0 \circ g_0$ for all $n \geq 0$.  This contradicts the fact that $\{H_{\alpha}^{\Aff} g_n\}$ are all distinct cosets.

To prove the claim, notice that in the first two cases, a finite index subgroup of $D_\alpha(\SAff(X_\alpha,q_\alpha))$ is dense in the Lie subgroup $G \leq \PSL_2(\mathbb R)$, and $\Delta < G$ is a $1$--dimensional submanifold of $G$, which itself has dimenion $3$ or $2$ in cases (1) and (2), respectively.
This implies that there exists a sequence $\{ g_n \} \in \SAff(X_\alpha,q_\alpha)$ such that $D_{\alpha}(g_n) \rightarrow I$ as $n \rightarrow \infty$ but $D_{\alpha}(g_n) \notin \Delta$.
By way of contradiction, suppose that there exists a subsequence $\{ g_{n_i} \}$ such that $g_{n_i}$ are in the same coset $H_{\alpha}^{\Aff}g$ where $D_{\alpha}(g) \notin \Delta$.
This implies that $D_{\alpha}(g_{n_i}) \subset \Delta D_{\alpha}(g)$, which is a $1$--manifold parallel to $\Delta$ and does not accumulate to $I$. 
This contradicts the fact that $D_{\alpha}(g_{n_i}) \rightarrow I$. 
Therefore, there exists a subsequence of $\{ g_n \}$ such that $\{H_{\alpha}^{\Aff} g_n\}$ are all distinct cosets.


To prove the claim in the final case, we note that by assumption there exists a sequence of distinct cosets $H_\alpha^\Aff b_n^\Aff$ of $H_\alpha^\Aff$ in $\SAff(X_\alpha,q_\alpha)$.  Since both $H_\alpha^D$ and $D_\alpha(\SAff(X_\alpha,q_\alpha))$ are dense in $\Delta$, so is every coset of $H_\alpha^D$.  Therefore, we can find a sequence $\{a_n^\Aff\} \subset H_\alpha^\Aff$ so that $D_\alpha(a_n^\Aff)D_\alpha(b_n^\Aff)\to I$ as $n \to \infty$.  Let $g_n = a_n^\Aff b_n^\Aff$, so that $D_\alpha(a_n^\Aff) \to I$ and $H_\alpha^\Aff g_n$ are distinct cosets of $H_\alpha^\Aff$, as required. This completes the proof of the claim.  Since we already proved the proposition assuming the claim, we are done.
\end{proof}

To complete the proof of Theorem~\ref{T:lonely leaves}, we need only prove the following.
\begin{proposition}
$ \mathrm{Aff}_+(X_{\alpha}, q_{\alpha}) = \mathrm{SAff}(X_{\alpha}, q_{\alpha})$.
\end{proposition}

\begin{proof}  First, observe that $\SAff_+(X_\alpha,q_\alpha)$ is a normal subgroup of $\Aff_+(X_\alpha,q_\alpha)$ since it is precisely the kernel of the homomorphism given by the determinant of the derivative.  In fact, from this homomorphism, either $\Aff_+(X_\alpha,q_\alpha) = \SAff(X_\alpha,q_\alpha)$ or else the index is infinite; $[\Aff_+(X_\alpha,q_\alpha): \SAff(X_\alpha,q_\alpha)] = \infty$.

After passing to a finite index subgroup, $\Gamma < \Aff_+(X_\alpha,q_\alpha)$, if necessary, the conjugation action of $\Gamma$ on $\SAff_+(X_\alpha,q_\alpha)$ preserves the finite index subgroup $H_\alpha^\Aff$ (and without loss of generality, $H_\alpha^\Aff < \Gamma$).  It thus suffices to prove $\Gamma < \SAff_+(X_\alpha,q_\alpha)$, or equivalently, $D_\alpha(\Gamma) < \PSL_2(\mathbb R)$.  

Consider any element
\[ g = \begin{pmatrix}
a & b\\
c & d
\end{pmatrix}  \in D_\alpha(\Gamma) \quad \mbox{ and } \quad h = \begin{pmatrix}
\lambda & 0\\
0 & \lambda^{-1}
\end{pmatrix} \in H_\alpha^D,\]
with $\lambda \neq \pm 1$.
Then $ghg^{-1} \in H_\alpha^D$, and is given by
\begin{equation*}
\begin{aligned}
ghg^{-1} & =
    \frac{1}{\mathrm{det}(g)}\begin{pmatrix}
a & b\\
c & d
\end{pmatrix}
\begin{pmatrix}
\lambda & 0\\
0 & \lambda^{-1}
\end{pmatrix}
\begin{pmatrix}
d & -b\\
-c & a
\end{pmatrix}\\ &= \frac{1}{\mathrm{det}(g)} 
\begin{pmatrix}
ad \lambda - bc \lambda^{-1} & ab(\lambda-\lambda^{-1})\\
cd (\lambda - \lambda^{-1}) & ad\lambda^{-1}-bc \lambda
\end{pmatrix}.
\end{aligned}
\end{equation*}
In order for this element to be in $H_\alpha^D$ (hence diagonal), we must have that $ab = 0$ and $cd = 0$.
Suppose that $a = 0$. If $c = 0$, then we have the zero matrix, so we must have that $c \neq 0$ and instead that $d = 0$. This gives us that $g$ is a matrix of the form
\begin{equation*}
g = \begin{pmatrix}
0 & b\\
c & 0
\end{pmatrix}.
\end{equation*}
We note that the square of a matrix of this form is a diagonal matrix.
Similarly, if $b = 0$, we must have that $c = 0$ and we have that $g$ is a matrix of the form
\begin{equation*}
g=
\begin{pmatrix}
a & 0\\
0 & d
\end{pmatrix}.
\end{equation*}
Together, these two conclusions imply that either $g$ or $g^2$ is diagonal.

Now we show that $D_\alpha(\Gamma) < \PSL_2(\mathbb R)$.  If not, then there exists $g \in D_\alpha(\Gamma)$ with $0<\det(g) \neq 1$.  After squaring and inverting if necessary, we may assume that $g$ is diagonal,
\[ g= \left( \begin{matrix} \lambda & 0 \\ 0 & \sigma \end{matrix} \right), \]
and $0 < \det(g) = \lambda \sigma < 1$. Without loss of generality, suppose $\lambda < 1$.
Notice that there exists an element $h \in H_{\alpha}^D$ such that
 \begin{equation*}
 h = 
\begin{pmatrix}
\mu & 0\\
0 & \mu^{-1}
\end{pmatrix}
\end{equation*}
and there exist $n,k \in \mathbb{Z}$ so that
 \begin{equation*} m = g^{n}h^k =
\begin{pmatrix}
r & 0\\
0 & s
\end{pmatrix}
\end{equation*}
where $0 < r,s < 1$.
Therefore, $m^j$ is a contraction for all $j > 0$, which implies that it is contracting in both directions.  Fixing a saddle connection $\omega$ of $q_\alpha$, it follows that the length of $m^j(\omega)$ tends to $0$ as $j \to \infty$. This contradicts Corollary~\ref{C:bounded geometry}, part (1), and thus proves that $g \in \PSL_2(\mathbb R)$, as required.  
\end{proof}
\begin{remark} The final contradiction in the proof also follows from Theorem~1.1 of \cite{PrScVa}, since $D_\alpha(\Aff_+(X_\alpha,q_\alpha))$ is necessarily of type (i) in that theorem.
\end{remark}

\bibliographystyle{alpha}

\newcommand{\etalchar}[1]{$^{#1}$}

\end{document}